\documentclass{amsart}
\usepackage{amssymb,amsmath,amsthm}
\usepackage[dvips]{graphicx}
\usepackage{enumerate}

\let\oldmarginpar\marginpar
\renewcommand\marginpar[1]{\-\oldmarginpar[\raggedleft\footnotesize #1]%
{\raggedright\footnotesize #1}}

\begin{document}

\newtheorem{theorem}{Theorem}[section]
\newtheorem{corollary}[theorem]{Corollary}
\newtheorem{lemma}[theorem]{Lemma}
\newtheorem{proposition}[theorem]{Proposition}
\theoremstyle{definition}
\newtheorem{definition}[theorem]{Definition}
\theoremstyle{remark}
\newtheorem{remark}[theorem]{Remark}
\theoremstyle{definition}
\newtheorem{example}[theorem]{Example}

\def\rank{{\text{rank}\,}}

\numberwithin{equation}{section}

\title[Warped products in Riemannian manifolds]{Warped products in Riemannian manifolds}

\author{Kwang-Soon Park}
\address{Department of Mathematical Sciences, Seoul National University, Seoul 151-747, Republic of Korea}
\email{parkksn@gmail.com}



\begin{abstract}
In this paper we prove two inequalities relating the warping function to various curvature terms, for warped products isometrically immersed in Riemannian manifolds. This extends work of B. Y. Chen for the case of immersions into space forms. Finally we give an application where the target manifold is the Clifford torus.
\end{abstract}

\maketitle
\section{Introduction}\label{intro}
\addcontentsline{toc}{section}{Introduction}

Let $(B, g_B)$ and $(F, g_F)$ be Riemannian manifolds, where $g_B$ and $g_F$ are Riemannian metrics on manifolds $B$ and $F$, respectively.
Let $f$ be a positive differentiable function on $B$. Consider the product manifold $B\times F$ with the natural projections $\pi_1 : B\times F \mapsto B$
 and $\pi_2 : B\times F \mapsto F$. The {\it warped product manifold} $M=B\times_f F$ is the product manifold $B\times F$ equipped with the Riemannian
 metric $g$ such that
$$
|| X ||^2 = || d\pi_1 X ||^2 + f^2(\pi_1(x)) || d\pi_2 X ||^2
$$
for any tangent vector $X\in T_x M$, $x\in M$. Thus, we get $g = g_B + f^2g_F$. The function $f$ is called the {\it warping function} of the warped product
manifold $M$ \cite {C6}.

As we know, warped product manifolds play important roles in differential geometry and in physics, particularly in general relativity. And there are
lots of papers on this topic (\cite {C6}, references therein). According to the result of J. F. Nash \cite {N}, which says that every Riemannian manifold can be isometrically embedded in some Euclidean space, every warped product can be isometrically embedded in some Euclidean space. As a generalization of B. Y. Chen's
results (\cite {C2}, \cite {C3}), in this paper we obtain the main results: Theorem \ref {thm1}, Theorem \ref {thm2}.

The paper is organized as follows. In section 2 we recall some notions, which are needed in the following sections. In section 3 we obtain the main results
from which we can see both the upper bound and the lower bound of the function $\frac{\triangle f}{f}$.
In section 4 we give some applications.

\section{Preliminaries}\label{prelim}

Let $(M, g)$ be a $m$-dimensional Riemannian manifold and $N$ a $n$-dimensional submanifold of $(M, g)$. Denote by $\nabla$ and $\overline{\nabla}$
the Levi-Civita connections of $N$ and $M$, respectively. The {\it Gauss} and {\it Weingarten formulas} are given by
\begin{eqnarray}
  \overline{\nabla}_X Y & = & \nabla_X Y + h(X, Y), \label{eq: gauss} \\
  \overline{\nabla}_X Z & = & -A_Z X + D_X Z, \label{eq: weing}
\end{eqnarray}
respectively, for tangent vector fields $X,Y\in \Gamma(TN)$ and normal vector field $Z\in \Gamma(TN^{\perp})$, where $h$ denotes the {\it second fundamental
form}, $D$ the {\it normal connection}, and $A$ the {\it shape operator} of $N$ in $M$.

Then the second fundamental form and the shape operator are related by
\begin{equation}\label{eq: shape}
\langle A_Z X, Y \rangle = \langle h(X, Y), Z \rangle,
\end{equation}
where $\langle \ , \ \rangle$ denotes the induced metric on $N$ as well as the Riemannian metric $g$ on $M$. Choose a local orthonormal frame
$\{ e_1, \cdots, e_m \}$ of $TM$ such that $e_1, \cdots, e_n$ are tangent to $N$ and $e_{n+1}, \cdots, e_m$ are normal to $N$.

Then the {\it mean curvature vector} $\overrightarrow{H}$ is defined by
\begin{equation}\label{eq: mean}
\overrightarrow{H} = \frac{1}{n} trace \ h = \frac{1}{n}\sum_{i=1}^n h(e_i, e_i)
\end{equation}
and the {\it squared mean curvature} is given by $H^2 := \langle \overrightarrow{H}, \overrightarrow{H} \rangle$.

The {\it squared norm of the second
fundamental form} $h$ is given by
\begin{equation}\label{eq: sqnorm}
|| h ||^2 = \sum_{i,j=1}^n \langle h(e_i, e_j), h(e_i, e_j) \rangle.
\end{equation}
Let $h_{ij}^r := \langle h(e_i, e_j), e_r \rangle$ for $1\leq i,j\leq n$ and $n+1\leq r\leq m$.

A submanifold $N$ is said to be {\it totally geodesic} in $M$ if the second fundamental form of $N$ in $M$ vanishes identically. Denote by $K(\pi)$ and
$\overline{K}(\pi)$ the sectional curvatures of a plane $\pi \subset T_p N$, $p\in N$, in $N$ and in $M$, respectively. i.e., if the plane $\pi$ is spanned
by vectors $X,Y\in T_p N$, then we have
$$
K(\pi) = \frac{\langle R(X,Y)Y, X \rangle}{\langle X, X \rangle \cdot \langle Y, Y \rangle - \langle X, Y \rangle^2} \ \text{and} \
\overline{K}(\pi) = \frac{\langle \overline{R}(X,Y)Y, X \rangle}{\langle X, X \rangle \cdot \langle Y, Y \rangle - \langle X, Y \rangle^2},
$$
where $R$ and $\overline{R}$ are the Riemann curvature tensors of $N$ and $M$, respectively.

The {\it scalar curvature} $\tau$ of $N$ is defined by
\begin{equation}\label{eq: scalar}
\tau = \sum_{1\leq i<j\leq n} K(e_i\wedge e_j),
\end{equation}
where $K(e_i\wedge e_j) = \langle R(e_i,e_j)e_j, e_i \rangle$ for $1\leq i,j\leq n$.

Let
\begin{equation}\label{eq: inf}
(\inf \overline{K})(p) := \inf \{ \overline{K}(\pi) \mid \pi \subset T_p N, \dim \pi = 2 \}
\end{equation}
and
\begin{equation}\label{eq: sup}
(\sup \overline{K})(p) := \sup \{ \overline{K}(\pi) \mid \pi \subset T_p N, \dim \pi = 2 \}
\end{equation}
for $p\in N$.

The {\it Gauss equation} is given by
\begin{equation}\label{eq: gausseq}
R(X,Y,Z,W) = \overline{R}(X,Y,Z,W) + \langle h(X,W), h(Y,Z) \rangle - \langle h(X,Z), h(Y,W) \rangle
\end{equation}
for tangent vectors $X,Y,Z,W\in T_p N$, $p\in N$, where $R(X,Y,Z,W) = \langle R(X,Y)Z, W \rangle$ and
$\overline{R}(X,Y,Z,W) = \langle \overline{R}(X,Y)Z, W \rangle$.

Then we easily obtain
\begin{equation}\label{eq: scalar2}
\sum_{i,j=1}^n \overline{K}(e_i\wedge e_j) = 2\tau + || h ||^2 - n^2 H^2.
\end{equation}

The {\it Laplacian} of a differentiable function $f$ on $N$ is defined by
\begin{equation}\label{eq: lapl}
\triangle f = \sum_{i=1}^n ((\nabla_{e_i} e_i)f - e_i^2 f).
\end{equation}

\section{Some inequalities}\label{inequal}

Let $(B, g_B)$ be a $n_1$-dimensional Riemannian manifold and $(F, g_F)$ a $n_2$-dimensional Riemannian manifold with $n = n_1 + n_2$.
Let $(M, g)$ be a warped product manifold of $(B, g_B)$ and $(F, g_F)$ such that $M = B\times_f F$ and $g = g_B + f^2g_F$ with the projections
$\pi_1 : B\times F \mapsto B$ and $\pi_2 : B\times F \mapsto F$ and $(\overline{M}, \overline{g})$ a $m$-dimensional Riemannian manifold.
Let  $\mathcal{D}_1$ and $\mathcal{D}_2$ denote the distributions in $M$ obtained from the vectors tangent to the horizontal lifts of $B$ and $F$, respectively.

Let $\phi : (M, g)\mapsto (\overline{M}, \overline{g})$ be an isometric immersion. We choose a local orthonormal frame $\{ e_1, \cdots, e_n \}$ of
the tangent bundle $TM$ of $M$ such that $e_1, \cdots, e_{n_1}\in \Gamma(\mathcal{D}_1)$ and $e_{n_1+1}, \cdots, e_{n}\in \Gamma(\mathcal{D}_2)$.
Conveniently, we identify $d\phi(e_i)$ with $e_i$ for $1\leq i\leq n$.
We also choose a local orthonormal frame $\{ e_{n+1}, \cdots, e_m \}$ of the normal bundle $TM^{\perp}$ of $M$ in $\overline{M}$ via $\phi$ such that
$e_{n+1}$ is in the direction of the mean curvature vector field.

Then we have
\begin{equation}\label{eq: lapl2}
\triangle f = \sum_{i=1}^{n_1} ((\nabla_{e_i} e_i)f - e_i^2 f).
\end{equation}
Denote by $tr h_1$ and $tr h_2$ the trace of $h$ restricted to $B$ and $F$, respectively.
i.e.,
$$
tr h_1 = \sum_{i=1}^{n_1} h(e_i,e_i) \quad \text{and} \quad tr h_2 = \sum_{j=n_1+1}^{n} h(e_j,e_j).
$$
Given unit vector fields $X,Y\in \Gamma(TM)$ such that $X\in \Gamma(\mathcal{D}_1)$ and $Y\in \Gamma(\mathcal{D}_2)$, we easily obtain
\begin{equation}\label{eq: warp2}
\nabla_X Y = \nabla_Y X = (X \ln f)Y,
\end{equation}
where $\nabla$ is the Levi-Civita connection of $(M, g)$, so that
\begin{eqnarray}
K(X\wedge Y)
& = & \langle \nabla_Y\nabla_X X - \nabla_X\nabla_Y X, Y \rangle   \label{sect3}    \\
& = & \frac{1}{f} ((\nabla_X X)f - X^2 f).  \nonumber
\end{eqnarray}
Hence,
\begin{equation}\label{eq: sect4}
\frac{\triangle f}{f} = \sum_{i=1}^{n_1} K(e_i\wedge e_j)
\end{equation}
for each $j = n_1+1, \cdots, n$.

The map $\phi$ is called {\it mixed totally geodesic} if $h(e_i, e_j) = 0$ for $1\leq i\leq n_1$ and $n_1+1\leq j\leq n$.

Then we get

\begin{theorem}\label{thm1}
Let $(M = B\times_f F, g)$ be a warped product manifold of Riemannian manifolds $(B, g_B)$ and $(F, g_F)$ with the warping function $f$ and
$(\overline{M}, \overline{g})$ a Riemannian manifold. Let $\phi : (M, g)\mapsto (\overline{M}, \overline{g})$ be an isometric immersion. Then we obtain
\begin{equation}\label{eq: ineq1}
\frac{\triangle f}{f} \leq \frac{n^2}{4n_2} H^2 + n_1 \sup \overline{K},
\end{equation}
where $n_1 = \dim B$ and $n_2 = \dim F$ with $n=n_1+n_2$. The equality case of (\ref {eq: ineq1}) holds identically if and only if $\phi$ is a mixed totally
geodesic immersion such that $tr h_1 = tr h_2$ and $\overline{K}(X\wedge Y) = \sup \overline{K}$ for unit vectors $X\in \Gamma(\mathcal{D}_1)$ and $Y\in \Gamma(\mathcal{D}_2)$.
\end{theorem}

\begin{proof}
Given a local orthonormal frame $\{ e_1, \cdots, e_n \}$ of $TM$ such that $e_1, \cdots, e_{n_1}\in \Gamma(\mathcal{D}_1)$ and $e_{n_1+1}, \cdots, e_{n}\in \Gamma(\mathcal{D}_2)$, we have
\begin{equation}\label{eq: sect5}
\frac{\triangle f}{f} = \sum_{i=1}^{n_1} K(e_i\wedge e_j)
\end{equation}
for each $j = n_1+1, \cdots, n$.

By the Gauss equation, we get
\begin{equation}\label{eq: scalar3}
2\tau = n^2 H^2 - || h ||^2 + \sum_{i,j=1}^n \overline{K}(e_i\wedge e_j),
\end{equation}
where $\overline{K}(e_i\wedge e_j) = \overline{g}(\overline{R}(e_i,e_j)e_j, e_i)$ and $\overline{R}$ is the Riemann curvature tensor of $\overline{M}$.

Let
\begin{equation}\label{eq: delta}
\delta := 2\tau - \sum_{i,j=1}^n \overline{K}(e_i\wedge e_j) - \frac{n^2}{2} H^2.
\end{equation}

Then we obtain
\begin{equation}\label{eq: scalar4}
n^2 H^2 = 2\delta + 2|| h ||^2.
\end{equation}
Given a local orthonormal frame $\{ e_{n+1}, \cdots, e_m \}$ of the normal bundle such that
$e_{n+1}$ is in the direction of the mean curvature vector field, from (\ref {eq: scalar4}) we have
\begin{equation}\label{eq: scalar5}
(\sum_{i=1}^n h_{ii}^{n+1})^2 = 2\left (\delta + \sum_{i=1}^n (h_{ii}^{n+1})^2 + \sum_{i\neq j} (h_{ij}^{n+1})^2
+ \sum_{r=n+2}^m \sum_{i,j=1}^n (h_{ij}^{r})^2 \right ).
\end{equation}
Let $\displaystyle{a_1 := \sum_{i=1}^{n_1} h_{ii}^{n+1}}$ and $\displaystyle{a_2 := \sum_{i=n_1+1}^n h_{ii}^{n+1}}$.

Using the following relation
\begin{equation}\label{eq: rel1}
a_1^2 + a_2^2 \geq \frac{1}{2}(a_1 + a_2)^2,
\end{equation}
by (\ref {eq: scalar5}) we obtain
\begin{eqnarray}
&  & \sum_{1\leq j< k\leq n_1} h_{jj}^{n+1} h_{kk}^{n+1} + \sum_{n_1+1\leq s< t\leq n} h_{ss}^{n+1} h_{tt}^{n+1}   \label{eq: rel2}   \\
&  & \geq \frac{1}{2}\delta + \sum_{1\leq i<j \leq n} (h_{ij}^{n+1})^2 + \frac{1}{2} \sum_{r=n+2}^m \sum_{i,j=1}^n (h_{ij}^{r})^2.  \nonumber
\end{eqnarray}
From (\ref {eq: sect5}) and the Gauss equation (\ref {eq: gausseq}), we get
\begin{eqnarray}
& &\frac{n_2\triangle f}{f} = \tau - \sum_{1\leq i< j\leq n_1} K(e_i\wedge e_j) - \sum_{n_1+1\leq s< t\leq n} K(e_s\wedge e_t)  \label{eq: scalar6} \\
& &= \tau - \sum_{1\leq i< j\leq n_1} \overline{K}(e_i\wedge e_j) - \sum_{r=n+1}^m \sum_{1\leq i< j\leq n_1} (h_{ii}^r h_{jj}^r - (h_{ij}^r)^2)  \nonumber \\
& &- \sum_{n_1+1\leq s< t\leq n} \overline{K}(e_s\wedge e_t) - \sum_{r=n+1}^m \sum_{n_1+1\leq s< t\leq n} (h_{ss}^r h_{tt}^r - (h_{st}^r)^2)    \nonumber   \\
& &= \tau - \frac{1}{2}\sum_{i,j=1}^n \overline{K}(e_i\wedge e_j) + \sum_{\begin{subarray}{c} 1\leq i\leq n_1 \\  n_1+1\leq s\leq n \end{subarray}}
     \overline{K}(e_i\wedge e_s)      \nonumber   \\
& &- \sum_{r=n+1}^m \sum_{1\leq i< j\leq n_1} (h_{ii}^r h_{jj}^r - (h_{ij}^r)^2) \nonumber \\
& &- \sum_{r=n+1}^m \sum_{n_1+1\leq s< t\leq n} (h_{ss}^r h_{tt}^r - (h_{st}^r)^2). \nonumber
\end{eqnarray}
By (\ref {eq: delta}), (\ref {eq: rel2}), and (\ref {eq: scalar6}), we obtain
\begin{eqnarray}
& &\frac{n_2\triangle f}{f} \leq \tau - \frac{1}{2}\sum_{i,j=1}^n \overline{K}(e_i\wedge e_j) +
      \sum_{\begin{subarray}{c} 1\leq i\leq n_1 \\  n_1+1\leq s\leq n \end{subarray}} \overline{K}(e_i\wedge e_s)  \label{eq: rel3} \\
& &-\frac{1}{2}\delta - \sum_{\begin{subarray}{c} 1\leq i\leq n_1 \\  n_1+1\leq s\leq n \end{subarray}} (h_{is}^{n+1})^2 -
      \frac{1}{2} \sum_{r=n+2}^m \sum_{i,j=1}^n (h_{ij}^r)^2    \nonumber \\
& &- \sum_{r=n+2}^m \sum_{1\leq i< j\leq n_1} (h_{ii}^r h_{jj}^r - (h_{ij}^r)^2) \nonumber \\
& &- \sum_{r=n+2}^m \sum_{n_1+1\leq s< t\leq n} (h_{ss}^r h_{tt}^r - (h_{st}^r)^2) \nonumber  \\
& &\leq \tau - \frac{1}{2}\sum_{i,j=1}^n \overline{K}(e_i\wedge e_j) + n_1 n_2 \sup \overline{K} - \frac{1}{2}\delta  \nonumber \\
& &- \sum_{r=n+1}^m \sum_{\begin{subarray}{c} 1\leq i\leq n_1 \\  n_1+1\leq s\leq n \end{subarray}} (h_{is}^{r})^2
    - \frac{1}{2} \sum_{r=n+2}^m (\sum_{1\leq i\leq n_1} h_{ii}^r)^2   \nonumber \\
& &- \frac{1}{2} \sum_{r=n+2}^m (\sum_{n_1+1\leq j\leq n} h_{jj}^r)^2  \nonumber \\
& &\leq \tau - \frac{1}{2}\sum_{i,j=1}^n \overline{K}(e_i\wedge e_j) + n_1 n_2 \sup \overline{K} - \frac{1}{2}\delta    \nonumber \\
& &= \frac{n^2}{4} H^2 + n_1 n_2 \sup \overline{K}.    \nonumber
\end{eqnarray}
Therefore, we have
$$
\frac{\triangle f}{f} \leq \frac{n^2}{4n_2} H^2 + n_1 \sup \overline{K}.
$$
Similar with Theorem 1.4 of \cite {C2}, from (\ref {eq: rel2}) and (\ref {eq: rel3}), we get that the equality sign of (\ref {eq: ineq1}) holds
if and only if the immersion $\phi$ is mixed totally geodesic such that $tr h_1 = tr h_2$ and $\overline{K}(X\wedge Y) = \sup \overline{K}$
for unit vectors $X\in \Gamma(\mathcal{D}_1)$ and $Y\in \Gamma(\mathcal{D}_2)$. Notice that we can choose $e_1$ and $e_{n_1+1}$ such that the plane spanned
by $e_1$ and $e_{n_1+1}$ is equal to the plane spanned by $X$ and $Y$ so that $\overline{K}(e_1\wedge e_{n_1+1}) = \overline{K}(X\wedge Y)$.
\end{proof}

\begin{remark}\label{rmk1}
1. If $(\overline{M}, \overline{g})$ is a Riemannian manifold of constant sectional curvature $c$, then (\ref {eq: ineq1}) becomes $(1.2)$ of
Theorem 1.4 of \cite {C2}.

2. Let $\overline{M}_{m_1,m_2} := S^{m_1}(\sqrt{\frac{m_1}{m}})\times S^{m_2}(\sqrt{\frac{m_2}{m}})\subset S^{m+1}(1)$ be the Clifford torus,
where $m = m_1 + m_2$, $m\geq 4$, and $2\leq m_1\leq m-2$ (\cite {CCK}, \cite {L}). As we know, $\overline{M}_{m_1,m_2}$ is a compact minimal hypersurface in
$S^{m+1}(1)$ and has only two distinct principal curvatures $\sqrt{\frac{m_2}{m_1}}$, $-\sqrt{\frac{m_1}{m_2}}$ with multiplicities $m_1$, $m_2$,
respectively. The squared norm of the second fundamental form of $\overline{M}_{m_1,m_2}$ in $S^{m+1}(1)$ is equal to $m$ so that by using the
Gauss equation, the scalar curvature of $\overline{M}_{m_1,m_2}$ is equal to $\frac{m(m-2)}{2}$ (\cite {E}, \cite {SY}).

Moreover, if we take two unit vectors $e_1,e_2\in T_p \overline{M}_{m_1,m_2}$, $p\in \overline{M}_{m_1,m_2}$ such that
$A_p e_1 = \sqrt{\frac{m_2}{m_1}} e_1$ and $A_p e_2 = -\sqrt{\frac{m_1}{m_2}} e_2$, then given $x,y\in \mathbb{R}$ with $x^2 + y^2 = 1$ we have
\begin{equation}\label{eq: rel033}
\text{Ric} (xe_1 + ye_2) = m-1 - (\frac{m_2}{m_1}x^2 + \frac{m_1}{m_2}y^2 + \frac{m^2}{m_1 m_2}x^2y^2),
\end{equation}
where $\text{Ric}$ denotes the Ricci curvature of $\overline{M}_{m_1,m_2}$.

With a simple computation, we easily obtain
\begin{equation}\label{eq: rel0033}
\text{Ric} \geq m-1 - (\frac{m_2}{m_1} + \frac{m_1}{m_2})
\end{equation}
with equality holding if and only if $(x, y) = (\pm \sqrt{\frac{m_2}{m}}, \pm \sqrt{\frac{m_1}{m}})$ (\cite {CI}, \cite {Z}).

In particular, we know that $\overline{M}_{m_1,m_2}$ is not a constant curvature space but has only three types of sectional curvatures
$\{ 0, \frac{m}{m_1}, \frac{m}{m_2} \}$ so that when we consider an isometric immersion $\phi : M\mapsto \overline{M}_{m_1,m_2}$,
Theorem \ref {thm1} will be useful \cite {C2}.
\end{remark}

To consider the next Theorem, we need to introduce the following Lemma. From \cite {C1}, we get

\begin{lemma}\label{lem1}
Let $a_1, \cdots, a_n, c$ be any real numbers with $n\geq 2$ such that
\begin{equation}\label{eq: rel33}
(\sum_{i=1}^n a_i)^2 = (n-1) (\sum_{i=1}^n a_i^2 + c).
\end{equation}
Then
\begin{equation}\label{eq: rel4}
2a_1 a_2 \geq c
\end{equation}
with equality holding if and only if $a_1 + a_2 = a_3 = \cdots = a_n$.
\end{lemma}

\begin{theorem}\label{thm2}
Let $(M = B\times_f F, g)$ be a warped product manifold of Riemannian manifolds $(B, g_B)$ and $(F, g_F)$ with the warping function $f$ and
$(\overline{M}, \overline{g})$ a Riemannian manifold. Let $\phi : (M, g)\mapsto (\overline{M}, \overline{g})$ be an isometric immersion. Then we have
\begin{equation}\label{eq: ineq2}
\frac{\triangle f}{f} \geq \frac{n_1 n^2}{2(n-1)} H^2 - \frac{n_1}{2} || h ||^2 + n_1 \inf \overline{K},
\end{equation}
where $n_1 = \dim B$ and $n_2 = \dim F$ with $n=n_1+n_2$.
\end{theorem}

\begin{proof}
We choose a local orthonormal frame $\{ e_1, \cdots, e_n \}$ of $TM$ such that $e_1, \cdots, e_{n_1}\in$
$\Gamma(\mathcal{D}_1)$ and $e_{n_1+1}, \cdots, e_{n}\in \Gamma(\mathcal{D}_2)$ and a local orthonormal frame $\{ e_{n+1}, \cdots, e_m \}$ of the normal bundle such that
$e_{n+1}$ is in the direction of the mean curvature vector field.

Using the Gauss equation, we get
\begin{equation}\label{eq: scalar33}
2\tau = n^2 H^2 - || h ||^2 + \sum_{i,j=1}^n \overline{K}(e_i\wedge e_j).
\end{equation}
Let
\begin{equation}\label{eq: scalar7}
\delta := 2\tau - \frac{n^2(n-2)}{n-1} H^2 - \sum_{i,j=1}^n \overline{K}(e_i\wedge e_j).
\end{equation}
From (\ref {eq: scalar33}) and (\ref {eq: scalar7}), we obtain
\begin{equation}\label{eq: scalar8}
n^2 H^2 = (n-1) || h ||^2 + (n-1) \delta
\end{equation}
so that
\begin{equation}\label{eq: scalar9}
(\sum_{i=1}^n h_{ii}^{n+1})^2 = (n-1)\left (\sum_{i=1}^n (h_{ii}^{n+1})^2 + \sum_{i\neq j} (h_{ij}^{n+1})^2
+ \sum_{r=n+2}^m \sum_{i,j=1}^n (h_{ij}^{r})^2 + \delta \right ).
\end{equation}
Applying Lemma \ref {lem1} to (\ref {eq: scalar9}) with $a_1 = h_{11}^{n+1}$ and $a_2 = h_{n_1+1n_1+1}^{n+1}$, we have
\begin{equation}\label{eq: rel8}
2h_{11}^{n+1} h_{n_1+1n_1+1}^{n+1} \geq \sum_{i\neq j} (h_{ij}^{n+1})^2
+ \sum_{r=n+2}^m \sum_{i,j=1}^n (h_{ij}^{r})^2 + \delta
\end{equation}
so that
\begin{eqnarray}
& &K(e_1\wedge e_{n_1+1}) \geq \sum_{r=n+1}^m \sum_{j\in S_{1n_1+1}} \{ (h_{1j}^r)^2 + (h_{n_1+1j}^r)^2 \}  \label{eq: rel9}  \\
& &+ \frac{1}{2} \sum_{\begin{subarray}{c} i,j\in S_{1n_1+1} \\ i\neq j \end{subarray}} (h_{ij}^{n+1})^2
   + \frac{1}{2} \sum_{r=n+2}^m \sum_{i,j\in S_{1n_1+1}} (h_{ij}^{r})^2  \nonumber   \\
& &+ \frac{1}{2} \sum_{r=n+2}^m (h_{11}^r + h_{n_1+1n_1+1}^r)^2 + \frac{\delta}{2} + \inf \overline{K}  \nonumber   \\
& &\geq \frac{\delta}{2} + \inf \overline{K},  \nonumber
\end{eqnarray}
where $S_{1n_1+1} = \{ 1, \cdots, n \} - \{ 1, n_1+1 \}$.

Similarly, we obtain
\begin{equation}\label{eq: rel10}
K(e_i\wedge e_{n_1+1}) \geq \frac{\delta}{2} + \inf \overline{K}
\end{equation}
for $1\leq i\leq n_1$.

Since $\displaystyle{ K(e_i\wedge e_{n_1+1}) = \frac{1}{f} ((\nabla_{e_i} e_i)f - e_i^2 f) }$ for $1\leq i\leq n_1$, by (\ref {eq: lapl2}),
(\ref {eq: scalar8}), and (\ref {eq: rel10}), we get
\begin{eqnarray}
\frac{\triangle f}{f}
& & \geq \frac{n_1}{2} \delta + n_1 \inf \overline{K}    \label{eq: rel11}  \\
& &\geq \frac{n_1 n^2}{2(n-1)} H^2 - \frac{n_1}{2} || h ||^2 + n_1 \inf \overline{K}.  \nonumber
\end{eqnarray}
Therefore, we have the result.
\end{proof}

\begin{remark}
1. In a similar way with Theorem 1 of \cite {C3}, we can also give a condition for the equality sign of (\ref {eq: ineq2}) to be held,
which is just the same with the condition of Theorem 1 of \cite {C3} except the additional condition $\overline{K}(X\wedge Y) = \inf \overline{K}$
for $X\in \Gamma(\mathcal{D}_1)$ and $Y\in \Gamma(\mathcal{D}_2)$.

2. From Theorem \ref {thm1} and Theorem \ref {thm2}, we obtain both the upper bound and the lower bound of the function $\frac{\triangle f}{f}$
as follows:
\begin{equation}\label{eq: rel11}
\frac{n_1 n^2}{2(n-1)} H^2 - \frac{n_1}{2} || h ||^2 + n_1 \inf \overline{K} \leq \frac{\triangle f}{f} \leq \frac{n^2}{4n_2} H^2 + n_1 \sup \overline{K}.
\end{equation}
\end{remark}

\section{Applications}\label{app}

Let $M(c_1,c_2) := M(c_1) \times M(c_2)$ be the product manifold of Riemannian manifolds $M(c_1)$ and $M(c_2)$,
where $M(c_i)$ is a constant curvature space of constant sectional curvature $c_i$ for $i = 1,2$.

Then we know that $M(c_1,c_2)$ has only three types of sectional curvatures $\{ c_1,c_2,0 \}$.

Let $c := \min \{ c_1,c_2,0 \}$ and $\bar{c} := \max \{ c_1,c_2,0 \}$. 

Using Theorem \ref {thm1}, we easily get

\begin{corollary}\label{cor111}
Let $(B\times_f F, g)$ be a warped product manifold of Riemannian manifolds $(B, g_B)$ and $(F, g_F)$ with the warping function $f$
such that  $n_1 = \dim B$, $n_2 = \dim F$, and $n=n_1+n_2$.

Let $\phi$ be an isometric immersion from the warped product manifold $(B\times_f F, g)$
to the product manifold $M(c_1,c_2)$.

Then we have
\begin{equation}\label{eq: rel12}
\frac{\triangle f}{f} \leq \frac{n^2}{4n_2} H^2 + n_1 \bar{c}.
\end{equation}
\end{corollary}

\begin{remark}\label{rmk2}
Let $(B\times_f F, g)$ be a warped product manifold of Riemannian manifolds $(B, g_B)$ and $(F, g_F)$ with the warping function $f$
such that  $n_1 = \dim B$, $n_2 = \dim F$, and $n=n_1+n_2$.

Let $\phi$ be an isometric minimal immersion from the warped product manifold $(B\times_f F, g)$
to the product manifold $M(c_1,c_2)$.

Then we obtain
\begin{equation}\label{eq: rel122}
\frac{\triangle f}{f} \leq  n_1 \bar{c}.
\end{equation}
\end{remark}

By Remark \ref {rmk2}, we get

\begin{theorem}
Let $(M=B\times_f F, g)$ be a warped product manifold of Riemannian manifolds $(B, g_B)$ and $(F, g_F)$ with the warping function $f$
such that  $n_1 = \dim B$, $n_2 = \dim F$, and $n=n_1+n_2$ and let $M(c_1,c_2)$ be the product manifold of constant curvature spaces  
$M(c_1)$ and $M(c_2)$. 

Then there does not exist an isometric minimal immersion $\phi$ from the warped product manifold $(M, g)$
to the product manifold $M(c_1,c_2)$ such that $\frac{\triangle f}{f}(\pi_1(p)) > n_1 \bar{c}$ for some $p\in M$.
\end{theorem}

Using Theorem \ref {thm2}, we have

\begin{corollary}\label{cor112}
Let $(B\times_f F, g)$ be a warped product manifold of Riemannian manifolds $(B, g_B)$ and $(F, g_F)$ with the warping function $f$
such that  $n_1 = \dim B$, $n_2 = \dim F$, and $n=n_1+n_2$.

Let $\phi$ be an isometric immersion from the warped product manifold $(B\times_f F, g)$
to the product manifold $M(c_1,c_2)$.

Then we obtain
\begin{equation}\label{eq: rel13}
\frac{\triangle f}{f} \geq \frac{n_1 n^2}{2(n-1)} H^2 - \frac{n_1}{2} || h ||^2 + n_1 c.
\end{equation}
\end{corollary}

From Remark \ref {rmk1}, we know that the Clifford torus $\overline{M}_{m_1,m_2}$ is a product manifold of spheres $S^{m_1}(\sqrt{\frac{m_1}{m}})$ and
$S^{m_2}(\sqrt{\frac{m_2}{m}})$. i.e.,  $\overline{M}_{m_1,m_2}$ is a product manifold of constant curvature spaces 
$S^{m_1}(\sqrt{\frac{m_1}{m}})$ and $S^{m_2}(\sqrt{\frac{m_2}{m}})$, where $S^{m_1}(\sqrt{\frac{m_1}{m}})$ and $S^{m_2}(\sqrt{\frac{m_2}{m}})$ are 
constant curvature spaces of constant sectional curvatures $\frac{m}{m_1}$ and $\frac{m}{m_2}$, respectively. 

Thus, by using Corollary \ref{cor111}, we immediately obtain

\begin{corollary}\label{cor113}
Let $(B\times_f F, g)$ be a warped product manifold of Riemannian manifolds $(B, g_B)$ and $(F, g_F)$ with the warping function $f$
such that  $n_1 = \dim B$, $n_2 = \dim F$, and $n=n_1+n_2$.

Let $\phi$ be an isometric immersion from the warped product manifold $(B\times_f F, g)$
to the Clifford torus $\overline{M}_{m_1,m_2}$ with $2\leq m_1 \leq \frac{m}{2}$, $m = m_1 + m_2$, and $m\geq 4$.

Then we have
\begin{equation}\label{eq: rel12}
\frac{\triangle f}{f} \leq \frac{n^2}{4n_2} H^2 + \frac{n_1 m}{m_1}.
\end{equation}
\end{corollary}

\begin{remark}\label{rmk3}
Let $(B\times_f F, g)$ be a warped product manifold of Riemannian manifolds $(B, g_B)$ and $(F, g_F)$ with the warping function $f$
such that  $n_1 = \dim B$, $n_2 = \dim F$, and $n=n_1+n_2$.

Let $\phi$ be an isometric minimal immersion from the warped product manifold $(B\times_f F, g)$
to the Clifford torus $\overline{M}_{m_1,m_2}$ with $2\leq m_1 \leq \frac{m}{2}$, $m = m_1 + m_2$, and $m\geq 4$.

Then we obtain
\begin{equation}\label{eq: rel122}
\frac{\triangle f}{f} \leq  \frac{n_1 m}{m_1}.
\end{equation}
\end{remark}

By Remark \ref {rmk3}, we get

\begin{theorem}
Let $(M=B\times_f F, g)$ be a warped product manifold of Riemannian manifolds $(B, g_B)$ and $(F, g_F)$ with the warping function $f$
such that  $n_1 = \dim B$, $n_2 = \dim F$, and $n=n_1+n_2$ and let $\overline{M}_{m_1,m_2}$ be the Clifford torus $S^{m_1}(\sqrt{\frac{m_1}{m}})\times S^{m_2}(\sqrt{\frac{m_2}{m}})$ such that $2\leq m_1 \leq \frac{m}{2}$, $m = m_1 + m_2$, and $m\geq 4$.

Then there does not exist an isometric minimal immersion $\phi$ from the warped product manifold $(M, g)$
to the Clifford torus $\overline{M}_{m_1,m_2}$ such that $\frac{\triangle f}{f}(\pi_1(p)) > \frac{n_1 m}{m_1}$ for some $p\in M$.
\end{theorem}

Using Corollary \ref {cor112}, we have

\begin{corollary}
Let $(B\times_f F, g)$ be a warped product manifold of Riemannian manifolds $(B, g_B)$ and $(F, g_F)$ with the warping function $f$
such that  $n_1 = \dim B$, $n_2 = \dim F$, and $n=n_1+n_2$.

Let $\phi$ be an isometric immersion from the warped product manifold $(B\times_f F, g)$
to the Clifford torus $\overline{M}_{m_1,m_2}$ with $2\leq m_1 \leq \frac{m}{2}$, $m = m_1 + m_2$, and $m\geq 4$.

Then we obtain
\begin{equation}\label{eq: rel13}
\frac{\triangle f}{f} \geq \frac{n_1 n^2}{2(n-1)} H^2 - \frac{n_1}{2} || h ||^2.
\end{equation}
\end{corollary}

\section*{Acknowledgments}

The author is grateful to the referees for their valuable comments and suggestions.

\end{document}